\documentclass[11pt]{amsart}
\usepackage{amsmath,amsfonts, mathtools,amssymb,latexsym,amsthm,adjustbox,mathrsfs,amsbsy, bm,tikz-cd,indentfirst,makecell,graphicx, verbatim, calc, enumerate,xy, hyperref}
\usepackage[top=1in,bottom=1.3in,left=1.1in,right=1.1in]{geometry}
\hypersetup{colorlinks=true,
    linkcolor=blue,
    filecolor=magenta,
    urlcolor=cyan,}

\theoremstyle{plain}
\newtheorem{theorem}{Theorem}[section]
\newtheorem{lemma}[theorem]{Lemma}

\newtheorem{corollary}[theorem]{Corollary}
\newtheorem{proposition}[theorem]{Proposition}
\newtheorem{remark}[theorem]{Remark}
\newtheorem{definition}[theorem]{Definition}

\newtheorem{example}[theorem]{Example}
\newtheorem{point}[theorem]{}

\newcommand{\p}{\mathfrak{p}}
\newcommand{\m}{\mathfrak{m}}

\newcommand{\R}{\mathcal{R}}
\newcommand{\F}{\mathcal{F}}
\newcommand{\Z}{\mathbb{Z}}
\newcommand{\N}{\mathbb{N}}
\newcommand{\wt}{\widetilde}

\newcommand{\Ass}{\operatorname{Ass}}

\newcommand{\Spec}{\operatorname{Spec}}

\usepackage{csquotes}
\usepackage{graphicx}
\usepackage{subfigure}
\usepackage{epsfig}
\usepackage{epstopdf}
\usepackage{float}
\theoremstyle{plain}

\usepackage{xcolor}
\usepackage{titlesec}

\titleformat{name=\section}{}{\thetitle.}{0.5em}{\centering\scshape\large}

\begin{document}
\title{\large \textbf{ On unmixed and equidimensional associated graded rings }}
\author{\textsc{Tony J. Puthenpurakal}}
\email{tputhen@math.iitb.ac.in}

\author{\textsc{Samarendra Sahoo}}
\email{samarendra.s.math@gmail.com}

\address{Department of Mathematics, IIT Bombay, Powai, Mumbai 400 076, India}
\date{\today}


\begin{abstract}
   Let $(A,\m)$ be an analytically unramified  Noetherian local ring of dimension $d \geq 1$ with positive depth, $I$ a $\m$-primary ideal of $A$, and let $\overline{I}$ be the integral closure ideal of $I$. If $A$ is of characteristic $p > 0$, then let $I^*$ denote the tight closure of $I$. Let $G_I(A)=\bigoplus_{n\geq 0}I^n/I^{n+1}$ be the associated graded ring of $A$ with respect to $I$. Assume  $G_I(A)$ is unmixed and equidimensional.  We show that either the function $n\mapsto \lambda(\overline{I^n}/I^n)$ coincides with a polynomial of degree $d-1$ for all $n\gg0$ or $\overline{I^n}=I^n$ for all $n\geq 1.$  We prove an analogous result for the tight closure filtration if $A$ is of characteristic $p > 0$.  When $A$ is generalized Cohen-Macaulay and $I$ is generated by a standard system of parameters, we give bounds for the first Hilbert coefficients of the integral closure filtration of $I$ and the tight closure filtration of $I$.
\end{abstract}
\maketitle

\section{Introduction}
Let $A$ be a commutative Noetherian local ring of dimension $d$ and $I$ an ideal of $A$. Let $G_I(A) = \bigoplus_{n \geq 0} I^n/I^{n+1}$ be the associated graded ring of $A$ with respect to $I$.
There has been a lot of research on when $G_I(A)$ is Cohen-Macaulay or has high depth. In this paper, we show that if $G_I(A)$ is unmixed and equidimensional ($i.e. \operatorname{dim}({G}_I(A)/\mathfrak{p})=d$ for all $\mathfrak{p}\in \operatorname{Ass}({G}_I(A)))$ then it has lots of consequences.
The condition that $G_I(A)$ is equidimensional is a mild one; it holds, for instance, if $A$ is universally catenary and equidimensional, see \ref{UE}.

 Recall that $x\in A$ is said to be integral over $I$ if there exists a monic polynomial $f(t)=t^r+a_{1}t^{r-1}+\ldots +a_{r-1}t+a_r$, where $a_i\in I^i$ for all $1\leq i\leq r$ such that $f(x)=0.$ Set of all integral element over $I$ is denoted as $\overline{I}.$ By a theorem of D. Rees, \cite[Theorem 1.4]{Drees1961}, if $A$ is an analytically unramified local ring then the integral closure $\overline{\R}(I)=\bigoplus_{n\geq 0}\overline{I^n}t^n$ of the Rees algebra $\R=\bigoplus_{n\geq 0}I^nt^n$ in $A[t]$ is module finite over $\R=\bigoplus_{n\geq 0}I^nt^n$. $\lambda(M)$ denote the length of an $A$-module $M$.

We first prove:

\begin{theorem}[Theorem \ref{Thm:integralclosure}]
\label{second}
    Let $(A,\m)$ be an analytically unramified local ring of positive depth, $I$ be an $\m$-primary  ideal of $A.$ If $G_I(A)$ is unmixed and equidimensional then the function $n\mapsto \lambda(\overline{I^n}/{I^n})$ coincides with a polynomial of degree $d-1$ for all $n\gg0$ or $\overline{I^n}=I^n$ for all $n\geq 1.$
\end{theorem}
In fact, we prove Theorem \ref{second} for a more general class of filtrations and then deduce Theorem \ref{second}.
We give an example for $\ref{second}$  which shows that the hypothesis of unmixedness in Theorem \ref{second} cannot be dropped (see \ref{exam}).

Now, let $A$ be a ring of characteristic $p > 0$. If $I$ is an ideal in $A$, we set $I^*$ to be the tight-closure of $I$. We prove:
\begin{theorem}[Theorem \ref{Thm:tightclosure}]
\label{second-tight}
    Let $(A,\m)$ be an analytically unramified local ring of characteristic $p > 0$ with positive depth and $I$ an $\m$-primary ideal of $A.$ If $G_I(A)$ is unmixed and equidimensional then the function $n\mapsto \lambda((I^n)^*/{I^n})$ coincides with a polynomial of degree $d-1$ for all $n\gg0$ or $(I^n)^* =I^n$ for all $n\geq 1.$
\end{theorem}

An example where $G_I(A)$ is both unmixed and equidimensional occurs when $I$ is generated by a system of parameters in a Cohen-Macaulay local ring, in which case $G_I(A)$ is itself Cohen-Macaulay, see \cite[Theorem 1.1.8]{Bruns}. More generally, if $A$ is an analytically unramified, generalized Cohen-Macaulay ring of positive dimension, and if $I$ is generated by a standard system of parameters (for the definition of standard system of parameter, see \ref{def}) of $A$, then $G_I(A)$ is also unmixed and equidimensional, see Remark \ref{gcm} and \cite[Theorem 5.4]{NVT}.

Let $\overline{e}_1(I)$ denote the first Hilbert coefficient associated with the integral closure filtration of $I$, and let $e_1^*(I)$ denote the first Hilbert coefficient corresponding to the tight closure filtration of $I$ (for the definition of Hilbert coefficients, see \ref{HC}). We prove:

\begin{theorem}[Theorem \ref{gcm-main}]\label{gcm-body}
    Let $(A,\m)$ be a generalized Cohen-Macaulay and analytically unramified local ring of dimension $d\geq 1$. Let $I$ be a standard system of parameters. Then,
\begin{enumerate}
     \item $$\overline{e}_1({I}) \geq -\left(\sum_{j=1}^{d-1}\binom{d-2}{j-1}\lambda(H^j_{\m}(A))\right).$$
    \item If the characteristic of $A$ is $p>0$, $$e^*_1(I)\geq -\left(\sum_{j=1}^{d-1}\binom{d-2}{j-1}\lambda(H^j_{\m}(A))\right).$$

\end{enumerate}
   If equality holds, then in the first case, $A$ is a regular local ring and in the second case, $A$ is an $F$-rational (if $A$ is a homomorphic image of a Cohen-Macaulay local ring).  Here $H^j_\m(A)$ denotes the $j$-{th} local cohomology module of $A$ with respect to $\m$.
\end{theorem}

Let $I$ be an ideal of $A$. The saturation of $I$ is denoted as $I^{sat}$ and defined by $$I^{sat}=I:\m^{\infty}=\bigcup_{i\geq 0}(I:\m^i).$$ We then prove:

 \begin{theorem}[Theorem \ref{thm:saturated}]
 \label{third}
      Let $(A,\m)$ be an excellent normal local domain of dimension $d\geq 1$. Let $I$ be an ideal of $A$. If ${G}_{I}(A)$ is unmixed and equidimensional, then one of the following two statements holds:
    \begin{enumerate}[ \rm (1)]
        \item $(I^n)^{sat}=I^n$ for all $n\geq 1.$
        \item $\ell(I)=d$.
    \end{enumerate}
 \end{theorem}

We now briefly describe the content of the paper. In Section 2, we explain the necessary preliminaries. In Section 3, we will first prove some results needed for our main theorem. Finally, we sketch a proof of Theorem \ref{second}. In section 4, we give a proof of the theorem \ref{third} and prove that a result analogous to Theorem \ref{second} is true for some non $\m$-primary ideals. Finally, in section five we prove Theorem \ref{gcm-body}.

\section{Preliminaries}
Throughout this paper, all rings considered are Noetherian, and all modules considered, unless stated otherwise, are finitely generated. We will use \cite{Bruns} as a general reference.
Also, $\lambda(M)$ denotes the length of an $A$-module.
\begin{point}
 \normalfont
 Let $(A, \m)$ be a local ring.
   An element $x\in  I$ is called $A$-superficial with respect to $I$ if there exists $c\in \N$ such that for all $n \geq c$, $$(I^{n+1}:x)\cap I^c = I^n.$$ If depth$(I, A)>0$, then one can show that an $A$-superficial element is $A$-regular. Furthermore, in this case $$(I^{n+1}:x) = I^n \text{ for all } n\gg0.$$ Superficial elements exist if the residue field is infinite, see \cite[Page 86]{hccmm}.
\end{point}

\begin{point}\label{finitm}
\normalfont
 A filtration of ideals $\F=\{I_n\}_{n\geq 0}$ is said to be $I$-filtration if it satisfies the following conditions:
 \begin{enumerate}
     \item $I_0=A$ and $I_{n+1}\subseteq I_n$ for all $n\geq 0.$
     \item $I\subseteq I_1$ and $I_1\neq A.$
     \item $I_nI_m\subseteq I_{n+m}$ for all $n,m\geq 0$.
 \end{enumerate}
 In addition, if $II_n=I_{n+1}$ for all $n\gg0$ then the $I$-filtration $\F$ is called as an $I$-stable filtration.
\end{point}

It is well known that if $\F$ is an $I$-stable filtration, then  the Rees algebra $\R(\F)=\bigoplus_{n\geq 0}I_nt^n$ is a finite module over $\R=A[It]=\bigoplus_{n\geq 0}I^nt^n$ and hence $\R(\F)\subseteq \overline{\R}(I)=\bigoplus_{n\geq 0}\overline{I^n}t^n=$the integral closure Rees algebra of $\R.$

\begin{point}
    \normalfont
    A filtration of submodules $\{M_n\}_{n\geq 0}$ of an $A$-module $M$ is said to be $I$-filtration if it satisfies the following conditions:
    \begin{enumerate}
        \item $M_0=M$ and $M_{n+1}\subseteq M_n$ for all $n\geq 0$.
        \item $IM_n\subseteq M_{n+1}$ for all $n\geq 0.$

    \end{enumerate}
In addition, if $IM_n=M_{n+1}$ for all $n\gg0$ then the $I$-filtration $\{M_n\}_{n\geq 0}$ is an $I$-stable filtration.
\end{point}
\begin{point}
\normalfont
 Let $(A,\m)$ be a Noetherian local ring. Let $\F = \{I_n\}_{n\geq 0}$ be an $I$-stable filtration of $A$. Set ${L}^{\F}=\bigoplus_{n\geq 0}A/{I_{n+1}}$. Let $\R=\bigoplus_{n\geq 0}{I^n}t^n$ and  ${\R}(\F)=\bigoplus_{n\geq 0}{I_n}t^n$ be the Rees algebras with respect to the filtration $\{I^n\}_{n\geq 0}$ and $\{I_n\}_{n\geq 0}$, respectively. Since ${\R}(\F)$ is a subring of $A[t]$, the ring $A[t]$ is naturally a ${\R}(\F)$-module. We have the exact sequence of $\R(\F)$-modules $$0\to {\R}(\F)\to A[t]\to {L}^{\F}(A)(-1)\to 0 .$$ It follows that  ${L}^{\F}(A)(-1)$ is an $\R(\F)$-module. Moreover, since $\R\subseteq \R(\F)$, we see that ${L}^{\F}(A)$ is also an $\R$-module, for more details see \cite[4.2]{Part1}.
 Note, however, that $L^\F(A)$ is \emph{not} a finitely generated $\R(\F)$-module.
\end{point}

\begin{point}
\normalfont
Let $A$ be a ring and $I$ an ideal of $A$. Let $M$ be an $A$-module. Consider the following ascending chain of ideals in $A$: $$I\subseteq (IM:M)\subseteq (I^2M:IM)\subseteq (I^3M:I^2M)\subseteq \ldots \subseteq(I^{n+1}M:I^nM) \ldots .$$ Since $A$ is Noetherian, this chain stabilizes. The stable value is denoted by $r(I, M)$. It is easy to prove that the filtration of ideals $\mathcal{F}^I_M=\{r(I^n, M)\}_{n\geq 0}$ forms an $I$-filtration, see \cite[Theorem 2.1]{Zulfeqarr}.
\end{point}

Now, consider the following chain of submodules of $M$: $$IM\subseteq (I^2M:_MI)\subseteq (I^3M:_MI^2)\subseteq \ldots \subseteq (I^{n+1}M:_MI^n) \ldots .$$ Since $M$ is Noetherian, this chain of submodules stabilizes. The stable value is denoted by $\widetilde{IM}$. It can be proved that the filtration $\{\widetilde{I^nM}\}_{n\geq 0}$ is an $I$-filtration.

\begin{remark}
\label{rmk:tildeI^nM}
\begin{enumerate}
    \item If grade$(I, M)>0$, then for all $n\gg0$,  $I^nM=\widetilde{I^nM}.$ In particular, $\{\widetilde{I^nM}\}_{n\geq 0}$ is an $I$-stable filtration, see \cite[Lemma 2.1]{tildeI^nM}.
    \item If grade$(G_I(A)_+, G_I(M))>0$ then $I^nM=\widetilde{I^nM}$ for all $n\geq 1$, see \cite[fact 9]{Heinzer}.
\end{enumerate}
\end{remark}
Recall that an ideal $I$ of $A$ is said to be regular if it contains a non-zero divisor.
\begin{theorem}\cite[Theorem 5.9]{Zulfeqarr}
  \normalfont
  \label{Thm:Istable}
      Let $I$ be a regular ideal of $A.$ If ann$(M)=0$ then $\mathcal{F}^I_M=\{r(I^n, M)\}_{n\geq 0}$ is an $I$-stable filtration.
  \end{theorem}

 Let $\R=A[It]=\bigoplus_{n\geq 0}I^nt^n$ be the Rees ring. By Remark \ref{rmk:tildeI^nM} and Theorem \ref{Thm:Istable}, if $I$ is regular and ann$(M)=0$ then $\R({\F^I_M})=\bigoplus_{n\geq 0}r(I^n,M)t^n$ and $\wt{\R} =\bigoplus_{n\geq 0}\wt{I^n}t^n$ are finitely generated $\R$-module. Therefore, the quotient module ${\R({\F^I_M})}/{\wt{\R}}=\bigoplus_{n\geq 1}{r(I^n,M)}/{\wt{I^n}}$ is a finitely generated $\R$-module.

\begin{definition}\cite[Theorem 1]{kishor}
\normalfont
    Let $(A,\m)$ be a quasi-unmixed local ring of dimension $d\geq 1$ with infinite residue field. Let $I$ an $\m$-primary ideal of $A.$ Then there exist unique largest ideals $I_{(k)}$ for $1\leq k\leq d$ containing $I$ such that
    \begin{enumerate}
        \item $e_i(I)=e_i(I_{(k)})$ for $0\leq i\leq k$, where $e_i(I)$ is the $i$-th Hilbert coefficient of $G_I(A)$ (see \ref{HC}, for the definition of Hilbert coefficient).
        \item $I\subseteq I_{(d)}\subseteq \ldots \subseteq I_{(1)}\subseteq \overline{I}=\text{ integral closure of }I.$
    \end{enumerate}
    The ideal $I_{(k)}$ is called as $k$-th coefficient ideal of $A.$
\end{definition}
We now discuss the conditions for $G_I(A)$ to be unmixed and equidimensional.
\begin{proposition}
\label{UE}
     Let $(A,\m)$ be a Noetherian local ring of dimension $d$ and $I$ an ideal of $A.$
    \begin{enumerate}[\rm (1)]
     \item If $A$ is equidimensional and universal catenary then $G_{I}(A)$ is equidimensional.
        \item If $A$ is quasi unmixed with infinite residue field  and $I$ is $\m$-primary then  $G_{I}(A)$ is unmixed iff $I^n=I^n_{(1)}$ for every $n$, where $I^n_{(1)}$ is the first coefficient ideal of $I^n$.
    \end{enumerate}
\end{proposition}
\begin{proof}
     (1) The extended Rees ring $(\hat{\R}=A[It,t^{-1}])$ is equidimensional and catenary, see \cite[ Lemma 4.5.5]{Bruns}. Note that $t^{-1}$ is a non zero divisor of $\hat{\R}$. Let $\p \in {\Spec}(\R)$ be a minimal prime ideal of $(t^{-1}).$  Let $\mathfrak{M}$ be the graded maximal ideal of $\hat{\R}=A[It,t^{-1}]$.
     Then, by a graded analogue of \cite[Lemma 2, p.\ 250]{matsumura}, we have $\text{height}(\mathfrak{M}/\p) = \text{height}\ \mathfrak{M} - \text{height} \ \p = d+1 -1 = d$.
      Hence $G_I(A)$ is equidimensional.

      (2) For the proof of the second part, see \cite[Theorem 4]{kishor}.
\end{proof}

\begin{remark}\cite{Brodmann}
    Let $A$ be a Noetherian ring and $I$ an ideal of $A$. The set of associated primes $\operatorname{Ass}_A(A/I^n)$ stabilizes for sufficiently large $n$. The stable set is denoted by $A^*(I)$.
\end{remark}

\textbf{Notations:}
\begin{equation*}
    \begin{split}
       & \mathcal{F}^I_M=\{r(I^n,M)\}_{n\geq 0}, \,  \R(\F^I_M)=\bigoplus_{n\geq 0}r(I^n,M)t^n, \, \R=\bigoplus_{n\geq 0}I^nt^n, \, \wt{\R}=\bigoplus_{n\geq 0}\wt{I^n}t^n,   \\ & \F=\{I_n\}_{n\geq 0},  {\R}(\F)=\bigoplus_{n\geq 0}{I_n}t^n, \, {L}^{\F}(A)=\bigoplus_{n\geq 1}A/{I_{n+1}}, \, {G}_{\F}(A)=\bigoplus_{n\geq 0}{I_n}/{I_{n+1}}.
    \end{split}
\end{equation*}
For $\F=\{I^n\}_{n\geq 0}$, we write $L^{\F}(A)=L^I(A)$ and $G_{\F}(A)=G_I(A).$

\section{Polynomial growth of finitely generated Rees modules}

 \emph{This section assumes that $M_{\p}$ is free for all $\p \neq \m$ and the ring $(A,\m)$ has positive depth}. See \cite[section 6]{Zulfeqarr} for an example of such modules. We prove that if the associated graded ring of $A$ with respect to $I$ is unmixed and equidimensional, then the functions
$$n \mapsto \lambda(\overline{I^n}/I^n) \quad \text{and} \quad n \mapsto \lambda((I^n)^*/I^n)$$
are of polynomial type, and they attain their maximal degree whenever these functions are nonzero, where $\overline{I^n}$ and $(I^n)^*$ are integral closure and tight closure, respectively, of $I^n$ for all $n\geq 1$.

The authors proved the following theorem in \cite[Theorem 6.7]{Zulfeqarr}.

\begin{theorem}
 Let $I$ be a regular ideal of $A$. Then the function $n \to \lambda(r(I^n,M)/\widetilde{I^n})$ is of polynomial type. Furthermore, if $A^*(I) \cap \mathfrak{m}{\Spec}(A) = \phi$ then $r(I^n,M)=\widetilde{I^n}$ for all $n\geq 1$.
\end{theorem}

The following fact is well known. We give proof for the convenience of the reader.

\begin{lemma}
\label{Thm:upperbound}
\normalfont
Let $(A,\m)$ be a local ring. Let $\F=\{I_n\}_{n\geq 0}$ and $\F'=\{J_n\}_{n\geq 0}$ be two $I$-stable filtration of ideals of $A$ such that $J_n\subseteq I_n$ for all $n\geq 1$. Set $E=\R(\F)/\R{(\F')}=\bigoplus_{n\geq 0}I_n/J_n$. If $\lambda(E_n)<\infty$ for all $n\geq 1$ then the function $n\mapsto \lambda(E_n)$ is coincides with a polynomial of degree at most $\ell(I)- 1$ for all $n\gg0.$
\end{lemma}
\begin{proof}
   Since $\F$ and $\F'$ are $I$-stable filtrations, the Rees algebras $\R(\F)$ and $\R(\F')$ are finitely generated $\R$-modules. It follows that $E$ is a finitely generated $\R$-module. Since $\lambda(E_n)<\infty$ for all $n \geq 1$, there exists $s\in \N$ such that $\m^sE=0.$ Set $B_i=\m^iE/\m^{i+1}E$. It is clear that $B_i$ is a finitely generated graded $\R/\m \R$-module. Therefore, $\lambda((B_i)_n)$ coincides with the polynomial of degree at most $\ell(I)-1$ for $n\gg0.$ Now, we have the following filtration of submodules of $E$ $$E\supseteq \m E\supseteq \m^2 E \supseteq \ldots \supseteq \m^{s-1} E\supseteq 0 \ldots.$$ It follows that $\lambda(E_n)=\sum_{i=0}^{s-1}\lambda(\m^iE_n/\m^{i+1}E_n)$. Therefore, the function $n\mapsto \lambda(E_n)$ coincides with a polynomial of degree at most $ \ell(I)-1$ for all $n\gg0.$
    \end{proof}

To establish the main theorem, we require the following result, whose proof is similar to that of \cite[Proposition 5.6]{tonytrans}.

\begin{lemma}
\label{Thm:ASS}
    Let  $I$ be an ideal of $A$, $\F=\{I_n\}_{n\geq 0}$ be an $I$-stable filtration of $A$. Set $\R=A[It]$. Then $$\operatorname{Ass}_{\mathcal{R}}({G}_{\F}(A))=\operatorname{Ass}_{\mathcal{R}}({L}^{\F}(A)).$$
    \end{lemma}
    \begin{proof}
     Set $G={G}_{\F}(A)$ and $L={L}^{\F}(A)$.
      Since $G \subseteq L$, we can conclude that $\Ass_\R(G) \subseteq \Ass_\R(L)$.
      Let $L_s$ be an ${\R}(\F)$-submodule of $L$ defined as follows:

      $$L_s =\left \langle \bigoplus_{n=0}^{s}\dfrac{A}{{I_{n+1}}}\right \rangle.$$
It can be easily checked that the $n$-th graded component of $L_s$ is

\[
(L_s)_n =
\begin{cases}
\dfrac{A}{I_{n+1}}, & \text{if } n \leq s, \\[8pt]
\dfrac{I_{\,n-s}}{I_{\,n+1}}, & \text{if } n > s.
\end{cases}
\]

      Note that $L_s$ is a finitely generated ${\R}(\F)$-module and so of $\R$-module. It is clear that $L_s\subseteq L_{s+1}$ for all $s\geq 0$, and $L=\bigcup_{s\geq 0}L_s$. Thus, $$\operatorname{Ass}_{\R}(L)=\bigcup_{s\geq 0}\operatorname{Ass}_{\R}(L_s).$$ Also note that $L_0=G$. For all $s\geq 0$, we have an exact sequence $$0\to L_{s-1}\to L_s\to G(-s)\to 0.$$

For the case where $s=1$, we obtain the exact sequence $$0\to G\to L_1\to G(-1)\to 0.$$ This implies that $\operatorname{Ass}_{\R}(L_1)\subseteq \operatorname{Ass}_{\R}(G)$. Continuing this process, we obtain $\operatorname{Ass}_{\R}(L_s)\subseteq \operatorname{Ass}_{\R}(G)$ for all $s\geq 1.$ Therefore, $\Ass_\R (L) \subseteq \Ass_\R (G)$.
Hence $\operatorname{Ass}_{\R}(G)=\operatorname{Ass}_{\R}(L)$.
\end{proof}

 The next result is a key theorem of this paper.

\begin{theorem}\label{key}
     Let $(A,\mathfrak{m})$ be a Noetherian local ring of dimension $d\geq 1$ and $I$ be an ideal of $A$. Let $\F=\{I_n\}_{n\geq 0}$ $\F'=\{J_n\}_{n\geq 0}$ be two $I$-stable filtrations of $A$ such that $J_n\subseteq I_n$ for all $n\geq 1$. Set $E=\R(\F)/\R{(\F')}$. Assume that $\lambda(E_n)<\infty$ for all $n\geq 1$. If $G_{\F'}(A)$ is unmixed and equidimensional, then either $E=0$ or the function $n\mapsto \lambda(E_n)$ coincides with a polynomial of degree $d-1$ for all $n\gg0.$ Moreover, if $\ell(I)<d$ then $E=0.$
\end{theorem}

\begin{proof}
     Note that $E$ is a finitely generated $\R$-module. If $E=0$, then we have nothing to show. Suppose $E\neq 0$. By the Lemma \ref{Thm:ASS}, we get that $\operatorname{Ass}_{\mathcal{R}}({L}^{\F'}(A))=\operatorname{Ass}_{\mathcal{R}}({G}_{\F'}(A)).$ The short exact sequence $0 \to E \to {L}^{\F'}(A)(-1)$ yields $$\operatorname{Ass}_{\mathcal{R}}(E)\subseteq \operatorname{Ass}_{\mathcal{R}}({L}^{\F'}(A))=\operatorname{Ass}_{\mathcal{R}}({G}_{\F'}(A))$$ and $\operatorname{Min}_{\mathcal{R}}(E) \subseteq \operatorname{Ass}_{\mathcal{R}}(E) \subseteq \operatorname{Ass}_{\mathcal{R}}({G}_{\F'}(A)).$
    Since ${G}_{\F'}(A)$ is unmixed and equidimensional,  $\operatorname{dim}_{\mathcal{R}}(E) = \operatorname{dim}_{\mathcal{R}}({G}_{\F'}({A}))=d.$ This proves that $\lambda(E_n)$ coincides with a polynomial of degree $d-1$ for all $n\gg0.$

    Let $\ell(I)<d$. Assume that $E\neq 0$. Then $\operatorname{Ass}_{\mathcal{R}}(E)\neq \phi.$ This implies that the dimension of $E$ is $d$, which is a contradiction to Lemma \ref{Thm:upperbound}.
\end{proof}

\begin{remark}
   If $\dim A = 0$, then $\dim_{\R} E = 0$ in the Theorem \ref{key}; in particular, $I_n = J_n = 0$ for all sufficiently large $n$. Hence, it is more interesting to restrict our attention to the case $\dim(A) > 0$.
\end{remark}

We now establish an analogue of Theorem \ref{key} for the well-known $I$-stable filtrations.

\begin{theorem}
\label{Thm:main}
  Let $(A,\mathfrak{m})$ be a Noetherian local ring of dimension $d\geq 1$ and $I$ be a regular ideal of $A$. Let $M$ be an $A$-module with $M_{\p}$ is free for all $\p\neq \m.$ Set $E = {\R({\F^I_M})}/{\wt{\R}}=\bigoplus_{n\geq 1}{r(I^n,M)}/{\wt{I^n}}.$
  \begin{itemize}
      \item [(a)] Let $\wt{I^n}=I^n$ for all $n\geq l$. Set  $E_{\geq l}=\bigoplus_{n\geq l}r(I^n,M)/I^n $. If ${G}_I(A)$ is unmixed and equidimensional, then either $E_{\geq l}=0$ or the function $n\mapsto \lambda(E_n)$ coincides with a polynomial of degree $d-1$ for all $n\gg0.$
      \item[(b)] If $E_{\geq l}=0$ then $E=0.$
      \item[(c)] If $\ell(I)<d$ then $E=0.$
  \end{itemize}
  \end{theorem}
  \begin{proof}
(a)  Note that $E$ is a finitely generated $\mathcal{R}$-module (see Theorem \ref{Thm:Istable}), and hence so is $E_{\geq l}$. By the proof of \cite[Theorem 6.7]{Zulfeqarr}, we have $\lambda(E_n) < \infty$ for all $n \geq 1$. If $E_{\geq l} = 0$, then there is nothing to prove. Suppose instead that $E_{\geq l} \neq 0$. By the same argument as in the proof of Theorem~\ref{key}, we obtain
$$\dim_{\mathcal{R}}(E_{\geq l}) = \dim_{\mathcal{R}}(G_I(A)) = d.$$
It follows that $\lambda(E_n)$ coincides with a polynomial of degree $d-1$ for all $n \gg 0$.

     (b) By \cite[1.4, 2.4]{Part1} and \cite[Proposition 1.4(d)]{Zulfeqarr}, we may assume that the residue field of $A$ is infinite. Let $x\in I$ be an $A$-superficial element.
Since $M$ is a finitely generated $A$-module, we have the following surjective map $$A^{\mu(M)}\to M.$$ By \cite[Proposition 1.4(a),(b)]{Zulfeqarr}, $\wt{I^n}=r(I^n,A)=r(I^n,A^{\mu(M)})\subseteq r(I^n,M)$ for all $n\geq 1$. By Proposition 8.2 of \cite{Zulfeqarr}, we also have $\wt{I^n}:x=\wt{I^{n-1}}$ for all $n\geq 1$.

For $i<l$, we have $$x^{l-i}r(I^i,M)\subseteq r(I^l,M)=\wt{I^l}.$$ This implies that $\wt{I^i} \subseteq r(I^i,M)\subseteq (\wt{I^l} : x^{l-i}) = \wt{I^i}$. Therefore, $E_i=0$. Hence $E=0.$

(c) By part(b), it suffices to show that $E_{\geq l}=0$. Suppose  $E_{\geq l}\neq 0$. By part(a), the function $n\mapsto \lambda(E_n)$ coincides with a polynomial of degree $d-1$ for all $n\gg0.$ But this is a contradiction to Lemma \ref{Thm:upperbound}.
  \end{proof}

 We shall use the following result in the proof of Theorem \ref{Thm:integralclosure}.

\begin{lemma}
\label{Thm:closure}
    Let $(A,\m)$ be an analytically unramified local ring and $I$ be a regular ideal. Then there exists an $A$-module $M$ of rank one such that $r(I^n,M)=\overline{I^n}$ for all $n\geq 1.$
\end{lemma}
\begin{proof}
    By \cite[Theorem 4.4]{Zulfeqarr}, there exists a finitely generated module (say $T_n$) of rank one such that $r(I^n, T_n)=\overline{I^n}$ for each $n\geq 1.$ Since $A$ is analytically unramified, integral closure filtration is $I$-stable. So there exists $n_0$ such that $I\overline{I^n}=\overline{I^{n+1}}$ for all $n\geq n_0.$ Set $M=T_1\otimes T_2\otimes \ldots \otimes T_{n_{0}}.$ It is clear that $M$ is a finitely generated $A$-module of rank one. Therefore, by \cite[Proposition 4.2]{Zulfeqarr},  we get that $r(I^n,M)\subseteq \overline{I^n}$ for all $n\geq 1.$ This asserts that for all $1\leq n\leq n_0$,

    $$\overline{I^n}\supseteq r(I^n,M)\supseteq \sum_{i=1}^{n_0}r(I^n,T_i)\supseteq \overline{I^n},$$
    the second containment is by \cite[Proposition 1.4(c)]{Zulfeqarr}. Thus, $r(I^n,M)=\overline{I^n}$ for all $1\leq n\leq n_0.$ We also have  $\overline{I^{n_0+1}}=I\overline{I^{n_0}}=Ir(I^{n_0},M)\subseteq r(I^{n_0+1},M)\subseteq \overline{I^{n_0+1}}.$ This proves our result.
\end{proof}

\begin{remark}
\label{Thm:free}

\begin{enumerate}
\item  Let $(A,\m)$ be a local ring and let $J$ be an $\m$-primary ideal. Then $J$ is free on the punctured spectrum of $A$.
    \item If $I$ is an $\m$-primary ideal in the above lemma, then $M$ in the above lemma is the tensor product of $\m$-primary ideals of $A$, see \cite[Theorem 4.4]{Zulfeqarr}. In particular, $M_{\p}$ is free for all $\p \neq \m$.

\end{enumerate}

\end{remark}
\begin{theorem}
\label{Thm:integralclosure}
    Let $(A,\m)$ be an analytically unramified local ring of positive depth and $I$ an $\m$-primary ideal of $A.$ If $G_I(A)$ is unmixed and equidimensional then the function $n\mapsto \lambda({\overline{I^n}/{I^n}})$ coincides with a polynomial of degree $d-1$ for all $n\gg0$ or $\overline{I^n}={I^n}$ for all $n\geq 1$.
\end{theorem}
\begin{proof}
    As $(A,\m)$ is an analytically unramified ring of positive depth, by Lemma \ref{Thm:closure} and Remark \ref{Thm:free}(2), $r(I^n,M)=\overline{I^n}$ for all $n\geq 1$, where $M$ is the tensor product of $\m$-primary ideals of $A$. Note that the dimension of $G_I(A)$ is positive. Since  $G_I(A)$ is unmixed and equidimensional, the associated primes of $G_I(A)$ are minimal with height zero. This implies that the depth of $G_I(A)$ is positive with respect to the graded maximal ideal $\mathfrak{M}=\m/I\bigoplus_{n\geq 1}I^n/I^{n+1}$. Since $I$ is $\m$-primary, the degree zero component of $\mathfrak{M}$ contains only nilpotent elements. Therefore, $\text{grade}(G_I(A)_{+}, G_I(A))>0.$

    By Remark \ref{rmk:tildeI^nM}, we get $\wt{I^n}=I^n$ for all $n\geq 1.$ Set $E=\bigoplus_{n\geq 1}r(I^n,M)/\wt{I^n}=\bigoplus_{n\geq 1}\overline{I^n}/{I^n}$. Since $I$ is $\m$-primary, $\lambda(E_n)<\infty$ for all $n\geq 1$. By Theorem \ref{Thm:main}, either $\overline{I^n}={I^n}$ for all $n\geq 1$ or the function $n\mapsto \overline{I^n}/I^n$ coincides with a polynomial of degree $d-1$ for all $n\gg0$.
\end{proof}

\begin{remark}
Let $(A,\m)$ be a local ring of dimension $d\geq 1$ and $I$ an $\m$-primary ideal of $A$. If $G_I(A)$ is unmixed and equidimensional then grade$(I,A)>0.$
\end{remark}

\begin{definition}
\normalfont
     Let $(A,\m)$ be a local ring of characteristic $p>0$, $I$ an ideal of $A.$ The tight closure ideal of $I$, denoted by ${I}^*$, is  the set of all elements $x\in A$ for which there exists $c\in A^o$ with $cx^q\in I^{[q]}$ for $q\gg0$. Here $A^o$ is subset of $A$ which contains all the elements of $A$ outside all the minimal prime ideals of $A$, $q=p^e$ for some $e\in \N$ and    $I^{[q]}= \langle x^q  \mid x\in I\rangle$.
\end{definition}
Note that if $A$ is an analytically unramified local ring then the filtration $\{({I^n})^*\}_{n\geq 0}$ is an
 $I$-stable filtration as for all $n\geq 1$, $I^n\subseteq ({I^n})^* \subseteq \overline{I^n}$, see \cite[Proposition 10.2.5]{Bruns}.
\begin{theorem}
\label{Thm:tightclosure}
 Let $(A,\m)$ be an analytically unramified local ring of characteristic $p>0$ and $I$ an $\m$-primary ideal of $A.$ Let $\operatorname{dim}(A)=d\geq 1$.  If ${G}_{I}(A)$ is unmixed and equidimensional then either $({I^n})^*=I^n$ for all $n\geq 1$ or the function $n\mapsto\lambda(({I^n})^*/I^n)$ coincides with a polynomial of degree $d-1$ for all $n\gg0.$
\end{theorem}
\begin{proof}
    The tight closure filtration is an $I$-stable filtration since $(A,\m)$ is an analytically unramified ring; that is, the tight closure Rees algebra $\R^*(I)=\bigoplus_{n\geq 0}({I^n})^*t^n$ is a finitely generated module over the Rees algebra $\R.$ Thus, $E=\R^*(I)/\R=\bigoplus_{n\geq 1}({I^n})^*/I^n$ is a finitely generated $\R$-module. Since $I$ is $\m$-primary, $\lambda(({I^n})^*/I^n)<\infty$ for all $n\geq 1.$ The result now follows from Theorem \ref{key}.
\end{proof}

We now give some examples of Theorem \ref{Thm:integralclosure}. We used Macaulay2 \cite{M2} to verify the following examples.
\begin{example}
\normalfont
   Let $(A,\m)=(k[x,y]_{(x,y)}, (x,y)_{(x,y)})$ be a regular local ring and $I=(x^2,y^2)$ be an ideal of $A.$  By \cite[Proposition 12.1.1, 12.1.2]{Monomial}, if $I$ is a monomial ideal then so is $\overline{I}.$ In fact, $\overline{I}=\langle \alpha\in A\, |\, \alpha^l\in I^l \text{ for some $l>0$, $\alpha$ is a monomial}\rangle$. It can be checked that $I^n=(x^{2n},x^{2n-2}y^2,\ldots ,x^2y^{2n-2},y^{2n})$ and clearly the degree of generators of $I^n$ is $2n.$

   Let $x^ay^b\in \overline{I^n},$ i.e. there exists $l>0$ such that $(x^ay^b)^l\in I^{nl}$. It follows that $al+bl\geq 2nl\implies a+b\geq 2n.$ Since $(xy)^2\in I^2$ and $\text{deg}(x^n)=\text{deg}(y^n)=n<2n=\text{ degree of generators of } I^n$, so $x,y\notin \overline{I}.$ Therefore, $\overline{I}=\m^2$ and hence $\m^{2n}\subseteq \overline{I^n}$ for all $n\geq 1.$

    Suppose $x^ay^b\in \overline{I^n}$ with $a+b<2n$. Then there exists $l>0$ such that $al+bl\geq 2nl$, which is a contradiction. Therefore, $\overline{I^n}=m^{2n}$ for all $n\geq 1.$ Set $E=\bigoplus_{n\geq 1}\overline{I^n}/I^n.$ Since $G_I(A)$ is Cohen Macaulay, so by Theorem \ref{Thm:integralclosure}, $\lambda(E_n)$ is coincide with a polynomial of degree $1$ for all $n\gg0$. In fact, $\lambda(E_n)=n$ for all $n\geq 1$. Note that $\m^{2n+1}\subseteq I^n$ and $\m I^n=\m^{2n+1}$ for all $n\geq 1.$ So we get a exact sequence $$0\to I^n/\m I^n\to \m^{2n}/\m^{2n+1}\to \m^{2n}/I^n\to 0.$$ This implies that $\lambda(\m^{2n}/I^n)=\mu(\m^{2n})-\mu(I^n)=2n+1-(n+1)=n$ for all $n\geq 1.$
\end{example}

The next example says we cannot reject the hypothesis that $G_I(A)$ is unmixed from Theorem \ref{Thm:integralclosure}.

\begin{example}\label{exam}
    \normalfont
     Let $(A,\m)=(k[x,y]_{(x,y)}, (x,y)_{(x,y)})$ be a regular local ring and $I=(x^3,x^2y,y^3)$ be an ideal of $A.$ The associated graded ring $G_I(A)$ is not unmixed. Note that $\R/\m \R=\bigoplus_{n\geq 0}I^n/\m I^n$ is a Noetherian standard graded ring. Thus $\lambda(I^n/\m I^n)$ coincides with a polynomial for all $n\gg0$ (say $P_I(n)$) and $\ell(I)=\text{dim}(\R/\m \R)=\text{deg}(P_I(n))+1$. One can easily verify that $I^n$ contains all monomials of degree $3n$ except $xy^{3n-1}.$ Therefore, $\ell(I)=2.$

      Since $(xy^2)^2\in I^2$, $\m^3\subseteq \overline{I}$. By \cite[Corollary 12.1.6]{Monomial}, $\overline{I^n}$ is generated by monomials of degree $3n$ for all $n\geq 1$. Therefore, $\overline{I^n}=\m^{3n}$ for all $n\geq 1$. Note that $\m^{3n+1}\subseteq I^n \subseteq \m^{3n}$ and $\m I^n=\m^{3n+1}$ for all $n\geq 1.$ So we have the following exact sequence $$0\to I^n/\m I^n\to \m^{3n}/\m^{3n+1}\to \m^{3n}/I^n\to 0.$$ This implies that $\lambda(\m^{3n}/I^n)=\mu(\m^{3n})-\mu(I^n).$ Hence $\lambda(\overline{I^n}/I^n)=1$ for all $n\geq 1.$ This says that the function $n\mapsto \lambda(\overline{I^n}/I^n)$ neither zero nor coincides with a polynomial of degree $d-1=1.$
\end{example}
The unmixedness of $G_I(A)$ in the given examples was
verified with Macaulay2, by first computing $G_I(A)$ via  \sffamily{associatedGradedRing} and then it's associated primes via  \sffamily{associatedPrimes}.

\section{Non \texorpdfstring{$\m$-}-primary cases}
This section discusses that the Theorem \ref{Thm:integralclosure} is true even for non $\m$-primary ideals under some mild conditions.

We shall use the following result to prove the main theorem of this section.

\begin{lemma}\label{Length}
    Let $(A,\m)$ be an analytically unramified local ring of dimension $d\geq 2$ and $I$ be a radical ideal of $A$ of height $d-1$. Let $\F=\{I_n\}_{n\geq 0}$ be an $I$-stable filtration. If $A$ is an isolated singularity, then $\lambda(I_n/I^n)<\infty$ for all $n\geq 1$.
\end{lemma}

\begin{proof}

Since localization commutes with taking radicals. Thus, $I_{\p}$ is a radical ideal of $A_{\p}$.  The prime ideals of $A$ containing $I$ are $\m$ and the prime ideal of height $d-1$ containing $I$. Since the height of $I$ is $d-1$, localizing $I$ at the prime ideals of height $d-1$ is radical and $\p A_\p$-primary. Therefore,
$I_{\p}\cong \p A_\p.$

Since $\R(\F)\subseteq \overline{\R}(I)$ (see \ref{finitm}), it follows that $I_n/I^n$ is a submodule of  $\overline{I^n}/I^n$ for all $n\geq 1$. Therefore, it suffices to show that $\lambda(\overline{I^n}/I^n)<\infty$ for all $n\geq 1$, that is to show the associated primes of $\overline{I^n}/{I^n}$ contains only maximal ideal.

Let $V(I)$ be the set of all prime ideals of $A$ containing $I$. If we take $\p\in {\Spec}(A)\setminus V(I)$, then clearly $(\overline{I^n}/{I^n})_{\p}=0.$ Let $\p(\neq \m)$ be a prime ideal containing $I.$ As $A$ is an isolated singularity, so $A_{\p}$ is regular local ring for all $\p\neq \m$. Thus, $G_{\p A_{\p}}(A_{\p})$ is a domain for all $\p\neq \m$. By \cite[Lemma 3.1]{JKV}, $\overline{\p_{\p}^n}=\p_{\p}^n$ for all $n\geq 1.$ Hence $(\overline{I^n}/{I^n})_{\p}=\overline{\p_{\p}^n}/\p_{\p}^n=0$. This proves the result.

\end{proof}

\begin{theorem}
\label{Thm:finitelength}
Let $(A,\m)$ be an analytically unramified local ring of dimension $d\geq 2$ and $I$ be a radical ideal of $A$ of height $d-1$. If $A$ is isolated singularity and ${G}_{I}(A)$ is unmixed and equidimensional then either $\overline{I^n}={I^n}$  for all $n\geq 1$ or $\lambda(\overline{I^n}/I^n)$ coincides with a polynomial of degree $d-1$ for all $n\gg0.$ Moreover, if  $\ell(I)<d$ then $\overline{I^n}={I^n}$ for all $n\geq 1.$
\end{theorem}
\begin{proof}

  As $A$ is analytically unramified, we thus obtain the integral closure Rees algebra $\overline{\R}(I)=\bigoplus_{n\geq 0}\overline{I^n}t^n$ is a finitely generated $\R$-module. Set $E=\overline{\R}(I)/\R=\bigoplus_{n\geq 1}\overline{I^n}/{I^n}$. Clearly, $E$ is a finitely generated $\R$-module. By Lemma \ref{Length}, $\lambda(E_n)<\infty$ for all $n\geq 1$. The result now follows from the Theorem \ref{key}.

\end{proof}

The next application also follows from a result by Cowsik-Nori, see \cite[Proposition 3]{Cowsik}.
\begin{corollary}
      Let $(A,\m)$ be an analytically unramified local ring of dimension $d\geq 2$ and $\p$ be a prime ideal of $A$ of height $d-1$ with $\ell(\p)<d$. If $A$ is isolated singularity and ${G}_{\p}(A)$ is unmixed and equidimensional then $\overline{\p^n}={\p^n}$ for all $n\geq 1.$
\end{corollary}

\begin{definition}
\normalfont
    Let $(A,\m)$ be a Noetherian local ring and $I$ an ideal of $A.$ The saturation of $I$ is denoted as $I^{sat}$ and defined by $$I^{sat}=I:\m^{\infty}=\bigcup_{i\geq 0}(I:\m^i).$$ Note that $H^0_{\m}(A/I^n) = \bigcup_{i\geq 0}(0:_{A/I^n}\m^i)= (I^n)^{sat}/I^n$, where $H^0_{\m}(-)$ is the $0$-th local cohomology functor with respect to $\m$. Furthermore, if $I$ is $\m$-primary then $(I^n)^{sat}=A$ for all $n\geq 1$.
\end{definition}

In \cite[Theorem 2.1]{saturatedideal}, S. D. Cutkosky, J. Herzog and H. Srinivasan proved that if $(A,\m)$ is a excellent local domain and $I$ an ideal of $A$ with $\ell(I)<\text{dim}(A)$ then the saturated Rees algebra $\R^{sat}=\bigoplus_{n\geq 0}(I^n)^{sat}t^n$ is a finitely generated $\R$-module.
\begin{theorem}\label{thm:saturated}
    Let $(A,\m)$ be an excellent normal local domain of dimension $d\geq 1$. Let $I$ be an ideal of $A$. If ${G}_{I}(A)$ is unmixed and equidimensional, then one of the following two statements holds:
    \begin{enumerate}
        \item $(I^n)^{sat}=I^n$ for all $n\geq 1.$
        \item $\ell(I)=d$.
    \end{enumerate}
    \begin{proof}
        Suppose $\ell(I)<d$ and $(I^n)^{sat}\neq I^n$ for some $n\geq 1.$ By \cite[Theorem 2.1]{saturatedideal}, $\R^{sat}$ is a finitely generated $\R$-module. Set $E=\R^{sat}/\R$. Note that $\lambda((I^n)^{sat}/I^n)=\lambda(H^0_{\m}(A/I^n))<\infty$ all $n\geq 1$. Therefore, by Lemma \ref{Thm:upperbound}, $ \lambda(E_n)$ coincides with a polynomial of degree at most $\ell(I)-1\leq d-2.$ Since $E\neq 0$ and ${G}_{I}(A)$ is unmixed and equidimensional, the dimension of $E$ is equal to $d.$ This is a contradiction to the Theorem \ref{key}.
    \end{proof}
\end{theorem}

\section{Lower bounds to \texorpdfstring{$e^*_1(I)$}- and \texorpdfstring{$\overline{e}_1({I})$}-}
In this section, we provide lower bounds for $e^*_1(I)$ and $\overline{e}_1({I})$ in terms of the local cohomology of $A$. We also discuss the consequences when it attains its lower bound. To do this, we need to recall the following:

\begin{point}\label{HC}
    \normalfont

Let $\mathcal{F}=\{I_n\}_{n\geq 0}$ be an $I$-stable filtration, where $I$ is a $\m$-primary ideal of a Noetherian local ring $A$. Then the Rees algebra $\R(\mathcal{F})=\bigoplus_{n\geq 0}I_nt^n$ is a finitely generated $\R$-module. Consequently, the associated graded ring $G_{\mathcal{F}}(A)=\bigoplus_{n\geq 0}I_n/I_{n+1}$ is a finitely generated graded module over the associated graded ring $G_I(A)$. Therefore, the Hilbert-Samuel function defined by $H_{\mathcal{F}}(n)=\lambda(A/(I_{n+1})$ coincides with a polynomial, $P_{\mathcal{F}}(n)$, of degree $d=\text{dim }A$ for all $n\gg0.$ The polynomial can be expressed as:

$$P_{\mathcal{F}}(n)=e_0({\mathcal{F}})\binom{n+d}{d}-e_1({\mathcal{F}})\binom{n+d-1}{d-1}+\ldots +(-1)^de_d({\mathcal{F}}),$$
where $e_i({\mathcal{F}})\in \Z$ and it is called as the $i$-th Hilbert coefficient of $G_{\mathcal{F}}(A)$.

In particular:
\begin{itemize}
    \item[(1)] For the filtration $\{\overline{I^n}\}_{n\geq 0}$, the coefficients $\overline{e}_i(I)$ are called the {$i$-th normal Hilbert coefficients} (when $A$ is analytically unramified).
    \item[(2)] For the filtration $\{(I^n)^*\}_{n\geq 0}$, the coefficients $e_i^*(I)$ are called the {tight Hilbert coefficients} (when $A$ is analytically unramified of characteristic $p>0$).
    \item[(3)] For the $I$-adic filtration $\F=\{I^n\}_{n\geq 0}$, we denote the coefficients by $e_i(I)$, the {$i$-th Hilbert coefficients} of $G_I(A)$.
\end{itemize}
It is well known that $e_0(I)=e^*_0(I)=\overline{e}_0({I}).$

\end{point}
\begin{point}
\normalfont
    The ring $(A,\m)$ of dimension $d$ is said to be generalized Cohen-Macaulay if \\ $\lambda(H^i_{\m}(A))<\infty$ for all $i=0,\ldots ,d-1$, where $H^i_{\m}(A)$ is $i$-th local cohomology of $A$ with respect to $\m$.

    Let $q$ be a parameter ideal of $A.$ Set $I(q;A)=\lambda(A/q)-e(q;A)$, where $e(q;A)$ is multiplicity of $A$ with respect to $q$. Set $I(A)=\text{sup }I(q;A)$, where $q$ runs through all parameter ideals of $A.$ It is well known that if $A$ is generalized Cohen-Macaulay then $I(A)=I(q;A)$ for some parameter ideal $q.$
\end{point}
\begin{definition}
\normalfont
    Let $(A, \mathfrak{m})$ be a Noetherian local ring of dimension $d$. Then $A$ is called a \emph{Buchsbaum ring} if for every system of parameters $\mathbf{a} = a_1, \dots, a_d$ of $A$, the difference
\[
\lambda\big(A/{q}\big) - e(q; A)
\]
is independent of the choice of the system of parameters, where $q=(\mathbf{a})$.
\end{definition}
\begin{definition}\label{def}
    \normalfont
    Let $(A,\m)$ be a local ring of dimension $d.$ $a_1,\ldots ,a_d \in \m$ is said to be standard system of parameters of $A$ if $$I({a_1}^2,\ldots ,{a_d}^2 ; A)=I(q; A),$$ where $q=(a_1,\ldots ,a_d).$
\end{definition}

\begin{remark}\cite[Page 13]{NVT}
    If $(A,\m)$ is a Buchsbaum local ring, then every system of parameters is a standard system of parameters.
\end{remark}
\begin{proposition}\cite[Theorem 2.1]{NVT}
    Let $(A,\m)$ be a generalized Cohen-Macaulay ring of dimension $d.$ Then $a_1,
    \ldots ,a_d\in \m$ is standard system of parameters if $I(A)=I(q ; A),$ where $q=(a_1,\dots, a_d)$.
\end{proposition}

The following result examines the implications of the equality $\overline{e}_1(I)=e_1(I).$ A similar result was established in \cite[Theorem 1]{nvte1}.
\begin{theorem}
\label{Areg}
   Let $(A,\m)$ be a Cohen-Macaulay and analytically unramified local ring of dimension $d\geq 1$. Let $I$ be a system of parameters. If $\overline{e}_1(I)=e_1(I)$, then $A$ is a regular local ring.
\end{theorem}
\begin{proof}
For each $n\geq 0$, we have the following short exact sequence $$0\to \overline{I^{n+1}}/I^{n+1} \to A/I^{n+1}\to A/\overline{I^{n+1}}\to 0.$$  This induces $$0\to E  \to L^I(A)\to L^{\overline{I}}(A)\to 0,$$ where $E=\bigoplus_{n\geq 0}\overline{I^{n+1}}/I^{n+1}$, $L^I=\bigoplus_{n\geq 0}A/I^{n+1}$ and $L^{\overline{I}}=\bigoplus_{n\geq 0}A/\overline{I^{n+1}}.$ This implies for sufficiently large  $n$, $$\lambda(\overline{I^{n+1}}/I^{n+1})=(\overline{e}_1({I})- e_1(I))\binom{n+d-1}{d-1}+\text{ lower degree terms}.$$

Since $A$ is Cohen-Macaulay and $I$ is a system of parameters, $I$ is generated by a regular sequence. Therefore, $G_I(A)$ is Cohen-Macaulay. By Theorem \ref{Thm:integralclosure}, either $\overline{I^n}=I^n$ for all $n\geq 1$, or the function $n \mapsto \lambda(\overline{I^n}/I^n)$ coincides with a polynomial of degree $d-1$ for all $n\gg0$. Since $\overline{e}_1(I)=e_1(I)$, it follows that  $\overline{I^n}=I^n$ for all $n\geq 1$. The result then follows from Theorem 1.1 of \cite{Goto}.
\end{proof}

The following result explores the implications of the equality ${e}^*_1(I)=e_1(I).$ A related result was established in \cite[Corollary 4.6]{saipriya}, where the authors first reduced the problem to the Cohen–Macaulay case and then proved it.
\begin{theorem}
\label{Frational}
     Let $(A,\m)$ be a Cohen-Macaulay and analytically unramified local ring of dimension $d\geq 1$ with characteristics $p>0$.  Let $I$ be a system of parameters. If ${e}^*_1(I)=e_1(I)$, then $A$ is an $F$-rational ring.
\end{theorem}
\begin{proof}
    By the same arguments in Theorem \ref{Areg}, we obtain that $G_I(A)$ is Cohen-Macaulay and for sufficiently large  $n$, $$\lambda(({I^{n+1})^*}/I^{n+1})=({e}^*_1({I})- e_1(I))\binom{n+d-1}{d-1}+\text{ lower degree terms}.$$  Since $e^*_1({I})=e_1(I)$, it follows that from Corollary \ref{Thm:tightclosure} that  $(I^n)^*=I^n$ for $n \geq 1$. The result then follows from \cite[Proposition 10.3.5]{Bruns}.
\end{proof}

\begin{remark}\cite[Lemma 1.2(iv)]{NVT} \label{gcm}
    If $(A,\m)$ is a generalized Cohen-Macaulay local ring having positive depth, then $A$ is unmixed and equidimensional.
\end{remark}

\begin{theorem}\label{gcm-main}
    Let $(A,\m)$ be a generalized Cohen-Macaulay and analytically unramified local ring of dimension $d\geq 1$. Let $I$ be a standard system of parameters. Then,
\begin{enumerate}
     \item $$\overline{e}_1({I}) \geq -\left(\sum_{j=1}^{d-1}\binom{d-2}{j-1}\lambda(H^j_{\m}(A))\right).$$
    \item If the characteristic of $A$ is $p>0$, $$e^*_1(I)\geq -\left(\sum_{j=1}^{d-1}\binom{d-2}{j-1}\lambda(H^j_{\m}(A))\right).$$

\end{enumerate}
   If equality holds, then in the first case, $A$ is a regular local ring and in the second case, $A$ is an $F$-rational provided $A$ is a homomorphic image of a Cohen-Macaulay local ring.
\end{theorem}
\begin{proof}
By Theorem $4.1$ in \cite{NVT}, we have
\begin{equation*}
    \begin{split}
        \lambda(A/I^{n+1})= & \binom{n+d}{d}e_0(I)+ \sum_{i=1}^{d}\sum_{j=0}^{d-i}\binom{n+d-i}{d-i}\binom{d-i-1}{j-1}\lambda(H^j_{\m}(A))\\ = & \binom{n+d}{d}e_0(I)+ \binom{n+d-1}{d-1}\left[ \sum_{j=0}^{d-1}\binom{d-2}{j-1}\lambda(H^j_{\m}(A)) \right]+ \text{lower degree term.}
    \end{split}
\end{equation*}
for all $n\geq 0,$ where $\binom{d-i-1}{-1}=0$ if $i\neq d$ and $\binom{-1}{-1}=1.$ Since $A$ is analytically unramified, it has positive depth and by comparing the Hilbert coefficients, we obtain $$-e_1(I)=\sum_{j=1}^{d-1}\binom{d-2}{j-1}\lambda(H^j_{\m}(A)).$$

   On the other hand, since $A$ is a generalized Cohen-Macaulay local ring with positive depth, and by \cite[Theorem 5.4]{NVT}, $G_I(A)$ is a generalized Cohen-Macaulay ring with positive depth. Therefore, by Remark \ref{gcm}, $G_I(A)$ is unmixed and equidimensional. From the proof of Theorem \ref{Areg} and Theorem \ref{Frational}, we obtain

$$\overline{e}_1({I}) \geq  \,  e_1(I)= -\left(\sum_{j=1}^{d-1}\binom{d-2}{j-1}\lambda(H^j_{\m}(A))\right)$$ and if the characteristic of $A$ is $p>0$, $$e^*_1(I)\geq  \, e_1(I)= -\left(\sum_{j=1}^{d-1}\binom{d-2}{j-1}\lambda(H^j_{\m}(A))\right).$$

    Hence, by Theorem \ref{Areg} and Theorem \ref{Frational},  if equality holds, then in the first case $A$ is a regular local ring, and in the second case $A$ is an $F$-rational ring.
\end{proof}

\providecommand{\bysame}{\leavevmode\hbox to3em{\hrulefill}\thinspace}
\providecommand{\MR}{\relax\ifhmode\unskip\space\fi MR }
\providecommand{\MRhref}[2]{
  \href{http://www.ams.org/mathscinet-getitem?mr=#1}{#2}
}


\begin{thebibliography}{10}


\bibitem{Brodmann}
M. Brodmann, \emph{Asymptotic stability of $Ass(M/I^nM )$,} Proc. Amer. Math. Soc. \textbf{74} (1979), no. 1, 16–18. MR 521865

\bibitem{Bruns}
Winfried Bruns and J\"urgen Herzog, \emph{Cohen-Macaulay rings,} Cambridge Studies in Advanced Mathematics, vol. 39, Cambridge University Press, Cambridge, 1993. MR 1251956

\bibitem{Cowsik}
R. C. Cowsik and M. V. Nori, \emph{On the fibres of blowing up,} J. Indian Math. Soc. (N.S.) \textbf{40} (1976), no. 1-4, 217–222. MR 572990

\bibitem{saturatedideal}
Steven Dale Cutkosky, J\"urgen Herzog, and Hema Srinivasan, \emph{Asymptotic growth of algebras associated to powers of ideals,} Math. Proc. Cambridge Philos. Soc. \textbf{148} (2010), no. 1, 55–72. MR 2575372

\bibitem{saipriya}
Saipriya Dubey, Pham Hung Quy, and Jugal Verma, \emph{Tight Hilbert polynomial and F-rational local rings}, Res. Math. Sci.
\textbf{10} (2023), no. 1, Paper No. 8, 14. MR 4540882

\bibitem{M2}
David Eisenbud, Daniel R. Grayson, Michael Stillman, and Bernd Sturmfels (eds.), \emph{Computations in algebraic geometry with Macaulay2,} Algorithms and Computation in Mathematics, vol. 8, Springer-Verlag, Berlin, 2002. MR 1949544


\bibitem{Goto}
Shiro Goto, \emph{Integral closedness of complete-intersection ideals,} J. Algebra \textbf{108} (1987), no. 1, 151–160. MR 887198

\bibitem{Heinzer}
William Heinzer, Bernard Johnston, David Lantz, and Kishor Shah, \emph{The Ratliff-Rush ideals in a Noetherian ring: a survey,} Methods in module theory (Colorado Springs, CO, 1991), Lecture Notes in Pure and Appl. Math., vol. 140, Dekker, New York, 1993, pp. 149–159. MR 1203805


\bibitem{matsumura}
Hideyuki Matsumura, \emph{Commutative ring theory,} second ed., Cambridge Studies in Advanced Mathematics, vol. 8, Cambridge University Press, Cambridge, 1989, Translated from the Japanese by M. Reid. MR 1011461

\bibitem{JKV}
Mousumi Mandal, Shreedevi Masuti, and J. K. Verma, \emph{Normal Hilbert polynomials: a survey,} Commutative algebra and algebraic geometry (CAAG-2010), Ramanujan Math. Soc. Lect. Notes Ser., vol. 17, Ramanujan Math. Soc., Mysore, 2013, pp. 139–166. MR 3155958

\bibitem{nvte1}
M. Moral\'es, N. V. Trung, and O. Villamayor, \emph{Sur la fonction de {Hilbert-Samuel} des cl\^otures int\'egrales des puissances d’id\'eaux engendr\'es par un syst\'eme de param\'etres}, J. Algebra \textbf{129} (1990), no. 1, 96–102. MR 1037394

\bibitem{tildeI^nM}
Reza Naghipour, \emph{Ratliff-Rush closures of ideals with respect to a Noetherian module,} J. Pure Appl. Algebra \textbf{195} (2005), no. 2, 167–172. MR 2108469

\bibitem{hccmm}
Tony J. Puthenpurakal, \emph{Hilbert-coefficients of a Cohen-Macaulay module}, J. Algebra \textbf{264} (2003), no. 1, 82–97. MR 1980687

\bibitem{Part1}
Tony J. Puthenpurakal, \emph{Ratliff-Rush filtration, regularity and depth of higher associated graded modules. I,} J. Pure Appl. Algebra \textbf{208} (2007), no. 1, 159–176. MR 2269837

\bibitem{tonytrans}
Tony J. Puthenpurakal, \emph{A sub-functor for Ext and Cohen-Macaulay associated graded modules with bounded multiplicity}, Trans. Amer. Math. Soc. \textbf{373} (2020), no. 4, 2567–2589. MR 4069228

\bibitem{Zulfeqarr}
Tony J. Puthenpurakal and Fahed Zulfeqarr,
 \emph{Ratliff-Rush filtrations associated with ideals and modules over a Noetherian ring,} J. Algebra \textbf{311} (2007), no. 2, 551–583. MR 2314724


\bibitem{Drees1961}
D. Rees, \emph{A note on analytically unramified local rings,} J. London Math. Soc. \textbf{36} (1961), 24–28.
MR 126465

\bibitem{kishor}
Kishor Shah, \emph{Coefficient ideals,} Trans. Amer. Math. Soc. \textbf{327} (1991), no. 1, 373–384.
MR 1013338

\bibitem{NVT}
Ng\^o Vi\^et Trung, \emph{Toward a theory of generalized Cohen-Macaulay modules,} Nagoya Math. J. \textbf{102} (1986), 1–49. MR 846128

\bibitem{Monomial}
Rafael H. Villarreal, \emph{Monomial algebras,} second ed., Monographs and Research Notes in Mathematics, CRC Press, Boca Raton, FL, 2015. MR 3362802



\end{thebibliography}
\end{document}